\documentclass[11pt]{article}

\usepackage[a4paper,margin=2.6cm]{geometry}
\usepackage{microtype}
\usepackage{amsmath,amssymb,amsthm}
\usepackage{bm}
\usepackage{graphicx}
\usepackage{booktabs}
\usepackage[authoryear,round]{natbib}
\usepackage{xurl}
\usepackage[colorlinks=true,citecolor=blue,linkcolor=black,urlcolor=black]{hyperref}

\theoremstyle{plain}
\newtheorem{theorem}{Theorem}
\newtheorem{proposition}{Proposition}
\newtheorem{corollary}{Corollary}
\newtheorem{lemma}{Lemma}
\theoremstyle{definition}
\newtheorem{definition}{Definition}
\newtheorem{example}{Example}

\newenvironment{keywords}{\medskip\noindent\textit{Keywords}: }{\par}
\newcommand{\tbl}[2]{\caption{#1}\centering\small\setlength{\tabcolsep}{4pt}#2}
\renewcommand{\qed}{\hfill$\square$}

\title{Repairing covariance functions on metric trees\\ through a Hausdorff-type distance}
\author{Alfredo Alegr\'ia\\[4pt]
\normalsize Department of Statistics, Pontificia Universidad Cat\'olica de Chile\\
\normalsize Avenida Vicu\~na Mackenna 4860, Macul, Santiago 7820436, Chile\\
\normalsize \texttt{alfredo.alegria@uc.cl}}
\date{}

\begin{document}

\maketitle

\begin{abstract}
A compact metric tree is the natural domain of spatial data recorded along a river network, a root system, or the branched processes of a cell, and a suitable distance between two points is the length of the path joining them. That distance admits almost none of the covariance models of spatial statistics, since only the roughest members of each family stay nonnegative definite, and the resulting fields are nowhere mean square differentiable. The repair proposed here is a change of metric. Comparing two points through their distances to the leaves gives a Hausdorff-type metric built from the shortest-path one. The new metric has a closed form, agrees with the shortest-path one inside every edge, and is strictly smaller across branching points. Any correlation function valid in a Euclidean space of dimension equal to the number of leaves may then be composed with it, returning the Mat\'ern, powered exponential, and generalized Cauchy families at every degree of smoothness, and, on trees with few leaves, also correlations that change sign or have compact support. The same construction yields the multivariate Mat\'ern and space-time models of the Gneiting class. In simulations covering a range of smoothness and of spatial scale, maximum likelihood provides estimates without appreciable bias, and on a reconstructed microglial cell the smooth models reduce the cross-validated prediction error by half.
\end{abstract}


\section{Introduction}
\label{sec:intro}

Spatial data are not always recorded on a Euclidean domain. A river network \citep{ver2010moving}, a root system \citep{wu2016relationships}, and the branched processes of a cell \citep{abdolhoseini2019segmentation} are typical examples, since in each of them the observations lie along a one-dimensional set that branches but never closes on itself, a structure represented by a compact metric tree.

Analysing data on a metric tree requires the usual tools of spatial statistics, namely a distance between any two of its points, a natural one being the length of the path joining them, a second-order random field $Z$ indexed by the points of the tree, with a given covariance structure, likelihood estimation, and prediction at unsampled locations. These tools are well developed in Euclidean spaces, where standard textbooks (see, e.g., \citealp{cressie2015statistics}) provide a large catalogue of valid (nonnegative definite) parametric covariance families, many of which are radial, or isotropic, so that the covariance depends only on the Euclidean distance. Their validity follows from classical results such as that of \citet{schoenberg1938}.

That catalogue is not available on non-Euclidean domains. On a metric graph equipped with the shortest-path distance, and in particular on a metric tree, the domain of this paper, the families in common use remain nonnegative definite only under restrictions on their parameters that confine them to their roughest members \citep{anderes2020,moller2024cox,alegria2026computationally}, and the fields that survive are nowhere mean square differentiable. The same limitation appears when these classes are extended to more versatile settings, such as nonstationary models \citep{song2021modeling,alegria2026non} or space-time models \citep{porcu2023stationary,10.1214/23-EJS2206}.

One response to this limitation abandons the covariance function. In \cite{bolin2024}, a stochastic partial differential equation is solved on metric graphs, giving random fields of any smoothness without requiring a closed-form expression for the covariance function. Another route retains the covariance models and instead modifies the way distances between points are measured. In \cite{filosi2026smooth}, the shortest-path distance on metric graphs is replaced by a polyharmonic resistance distance, written as an infinite series built from the spectral representation of an operator defined on the graph.

We follow the second route because our interest is in a closed-form covariance function, which comes with a variety of parametric families and established methods for estimation and prediction, widely used across many disciplines, including the economic, environmental, and health sciences \citep{fischer2010handbook}. We aim to preserve this framework on metric trees.

The branching structure allows us to construct a distinctive metric. We compare two points through their distances to the leaves, that is, to the terminal vertices of the tree, rather than directly, which replaces the shortest-path metric by a Hausdorff-type metric built from it \citep{baddeley1992}. The new metric is simple to evaluate, and inside every edge it agrees with the shortest-path metric. At longer separations it contracts distances, while keeping the two metrics comparable, and that contraction restores nonnegative definiteness and opens a catalogue of covariance models that depends on the number of leaves.

Under the new metric, every tree admits the Mat\'ern, powered exponential, and generalized Cauchy families at any degree of smoothness, so that inside each edge the field is the corresponding Euclidean one. Trees with few leaves also admit correlations that change sign or have compact support. The same argument covers multivariate and space-time covariances, among them the multivariate Mat\'ern and the Gneiting class.

We restrict attention to metric trees for two reasons. A metric graph with cycles may have no leaves at all, and a cycle breaks the agreement with the shortest-path distance inside an edge, and with it the local behaviour of the field. We return to this case in Section~\ref{sec:discussion}.

In simulations we examine the accuracy of the maximum likelihood estimates and the loss from interchanging the two metrics. We then fit the models to a reconstructed microglial cell whose local radius is recorded at hundreds of points along its processes, where the smooth models reduce the prediction error by half.

The paper is organized as follows. The metric is introduced in Section~\ref{sec:trees}, together
with its properties. Section~\ref{sec:cov} builds the catalogue of covariance models, and
Sections~\ref{sec:sim} and~\ref{sec:data} apply it to simulated data and to a reconstructed
microglial cell. Section~\ref{sec:discussion} closes with remarks on graphs with cycles
and on other non-Euclidean domains, and the proofs are collected in the Appendix.

\section{An $L^2$ Hausdorff distance on metric trees}
\label{sec:trees}

\subsection{The metric and its structure}

By a metric tree, denoted by $\mathbb{T}$, we mean a connected union of finitely many segments of finite length that meet only at their endpoints and contain no cycle. It is a special case of a compact metric graph, in which each segment is an edge and its endpoints are vertices. The distance $d(s,t)$ between two points of $\mathbb{T}$ is the length of the unique path joining them. On a tree it coincides with the resistance distance of \citet{anderes2020}.

{The degree of a vertex is the number of edges that meet at it.} Write
$\xi_{1},\ldots,\xi_{\kappa}$ for the leaves of the tree, that is, its vertices of degree
one. A finite tree with at least one edge has at least two of them, so that $\kappa\geq2$.
The distance introduced below compares two points of $\mathbb{T}$ through their distances to
the leaves rather than directly.

\begin{definition}
\label{def:h}
Let $\mathbb{T}$ be a compact metric tree with leaves $\xi_{1},\ldots,\xi_{\kappa}$. The $L^2$
Hausdorff distance on $\mathbb{T}$ is
\begin{equation}
\label{eq:h}
h_\kappa(s,t) = \left( \frac{1}{\kappa} \sum_{j=1}^{\kappa} \,\big|d(\xi_{j},s) -
d(\xi_{j},t)\big|^2 \right)^{1/2},
\qquad s,t\in \mathbb{T}.
\end{equation}
\end{definition}

Expression \eqref{eq:h} is the order-two member of the error metrics of
\cite{baddeley1992}, which measure distances between binary images with the sum taken over
all pixels in the image raster. The version in Definition \ref{def:h} is tailored to metric
trees, since the sum is taken only over the leaves, and its properties are collected in the
following proposition.

\begin{proposition}
\label{thm:struct}
Let $\mathbb{T}$ be a compact metric tree with $\kappa$ leaves and let $h_\kappa$ be given by
\eqref{eq:h}. Then the following hold.
\begin{enumerate}
\item[\rm (i)] $h_\kappa$ is a metric on $\mathbb{T}$.

\item[\rm (ii)] The two metrics are comparable in both directions, in the sense that
\begin{equation}
\label{eq:bilip}
\sqrt{2/\kappa}\; d(s,t) \;\leq\; h_\kappa(s,t) \;\leq\; d(s,t), \qquad s,t\in \mathbb{T}.
\end{equation}

\item[\rm (iii)] Equality holds on the right of \eqref{eq:bilip} if and only if no vertex of
degree three or more lies strictly between $s$ and $t$. In particular $h_{\kappa}(s,t)=d(s,t)$
whenever $s$ and $t$ lie in a common edge.
\end{enumerate}
\end{proposition}

The proof is given in the Appendix. 

The choice of the leaves in \eqref{eq:h} is crucial. That expression could be written with any finite set of points of $\mathbb{T}$ in place of $\xi_{1},\hdots,\xi_{\kappa}$, and we call those the reference points. Each reference point $\xi$ enters through the term $|d(\xi,s)-d(\xi,t)|$. If the path from $\xi$ to $t$ passes through $s$, that term equals $d(s,t)$, and the same holds with
$s$ and $t$ interchanged. If instead it reaches the path from $s$ to $t$ at its midpoint, the term
is zero and $\xi$ carries no information about the pair.

The leaves are the reference points for which the full distance is always attained. Moving away from $s$ along any direction other than the one towards $t$, the path terminates at a leaf beyond $s$, and $s$ itself serves when it is a leaf. Doing the same at $t$ gives a leaf
beyond $t$. Every pair of points is therefore measured at full distance by at least two leaves, which is what produces the lower bound in part (ii) of Proposition \ref{thm:struct}. Neither removing nor adding reference points improves on this. With a single reference point $\xi$, \eqref{eq:h} reduces to $|d(\xi,s)-d(\xi,t)|$, which is zero for any two points at the same distance from $\xi$, even when they are far apart in the tree. Adding points to the leaves does not help either, since a reference point placed at a branching vertex contributes a zero term to the sum in \eqref{eq:h} for the two points that lie symmetrically about it, and lowers the average.

The map behind \eqref{eq:h} belongs to the family of Fr\'echet embeddings, which represent a point of a metric space by its distances to a fixed set of reference points \citep{linial1995geometry}. To the best of our knowledge, this construction has not been studied with the leaves of a metric tree as reference points, and the properties of the resulting metric have not been established. Changing the way distances are measured is a familiar device in spatial statistics, as in the deformation of \citet{sampson1992}, which maps the domain so that the field becomes isotropic in the new coordinates. Here the map is given in closed form by the tree itself.

A particular case of this device appears in the covariance models used on stream networks.
There each location is recorded by its upstream distance to the outlet, the most downstream
point of the network and itself a vertex of degree one
\citep{verhoef2006flow,ver2010moving}. Two locations are flow-connected when one of them lies on
the path of the other to the outlet, and flow-unconnected otherwise, and the two kinds of pair
are treated separately, which requires the edges to carry a direction. Averaging over all the
leaves handles the two kinds of pair in the same way, so that \eqref{eq:h} applies to a tree on
which no flow is defined, while a direction could be reintroduced by weighting the leaves
unequally.

\subsection{Comparison of the two metrics}

{Proposition \ref{thm:struct} describes} when the two metrics agree and the left-hand inequality of
\eqref{eq:bilip} says how much they can differ {at worst}. Between those two extremes the
contraction is governed by one number. We say that a leaf separates $s$ and $t$ when the path
coming from it reaches the path joining $s$ and $t$ at an interior point, rather than at $s$ or at
$t$, and we write $\kappa_{st}$ for the number of leaves that do so.

\begin{proposition}
\label{thm:local}
Let $\mathbb{T}$ be a compact metric tree with $\kappa$ leaves and let $s\neq t$. Then
\begin{equation}
\label{eq:local}
h_\kappa(s,t) \;\geq\; \big(1-\kappa_{st}/\kappa\big)^{1/2}\, d(s,t) .
\end{equation}
The left-hand inequality of \eqref{eq:bilip} is the case $\kappa_{st}=\kappa-2$, and the equality
of part (iii) is the case $\kappa_{st}=0$.
\end{proposition}

The proof is given in the Appendix. What governs the contraction is therefore the proportion of
leaves that separate the two points. In the extreme case $\kappa_{st}=\kappa-2$ the separating
leaves must reach the path at its midpoint, as in a star, while two points of a large tree whose
path crosses a few branching vertices have $\kappa_{st}$ small relative to $\kappa$ and are
measured almost as they are under the shortest-path metric.

Figure~\ref{fig:tree} uses a tree with $\kappa=4$ leaves placed so that two of them are close in the plane while far apart in the tree, and shows the distance from one of those two leaves, marked in white, to every other point of the tree under each metric. Several features can be noted. The leaf that lies next to the white point is the most distant point of the tree under $d$ and remains so under $h_\kappa$, so the ordering is the one given by the tree and not the one suggested by the plane. The scales differ, as Proposition \ref{thm:struct} states, and the right panel quantifies that difference. The spread of the cloud at a fixed value of $d$ is the range of $\kappa_{st}$ over the pairs at that distance, which is what Proposition \ref{thm:local} describes. Finally, the pairs lying within a single edge, in black, fall exactly on the diagonal, which is the last part of Proposition \ref{thm:struct}.
\begin{figure}
\centering
\includegraphics[width=\textwidth]{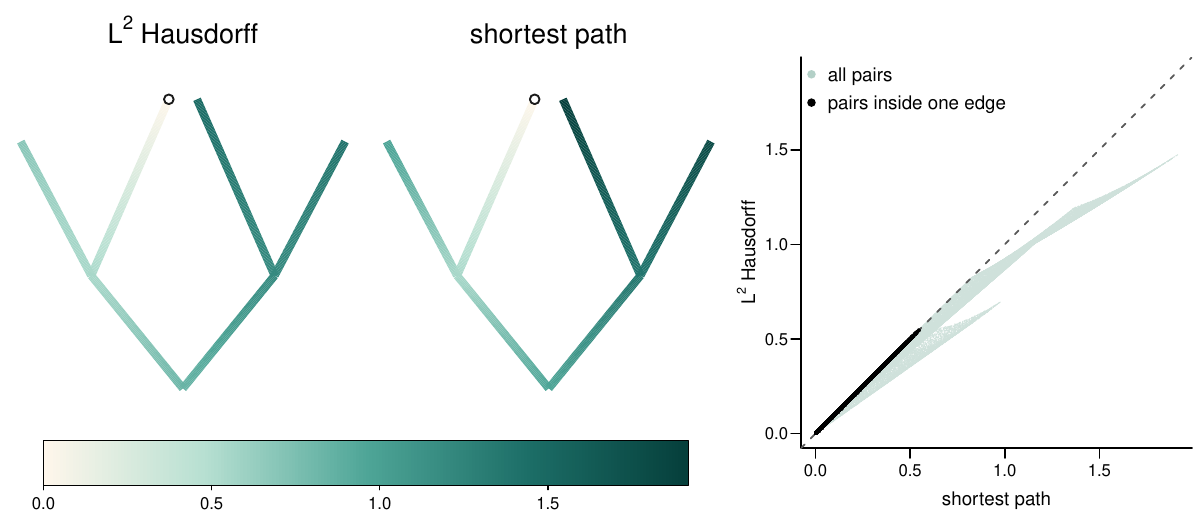}
\caption{A tree with $\kappa=4$ leaves. Left and
middle, the distance from the point marked by the white circle to every other point of the tree,
under $h_\kappa$ and under the shortest-path metric $d$. 
Right, $h_\kappa$ against $d$ over all pairs of points, with the pairs lying within a single edge
in black and the diagonal dashed.}
\label{fig:tree}
\end{figure}

\section{Catalogue of covariance functions on metric trees}
\label{sec:cov}

Throughout this section, $Z=\{Z(s),\,s\in\mathbb{T}\}$ denotes a real-valued random field on the
tree with finite second moments, and $C(s,t)=\mathrm{cov}\{Z(s),Z(t)\}$ denotes its covariance
function. Given a metric $\delta$ on $\mathbb{T}$, we say that $C$ is isotropic with respect to $\delta$
when $C(s,t)$ depends on the pair $(s,t)$ only through $\delta(s,t)$, which for $\delta=d$ is the
notion used by \cite{anderes2020}. The models constructed below are isotropic with respect to the
metric $h_\kappa$ of Section~\ref{sec:trees}, and in general they are not isotropic with respect to
the shortest-path metric $d$.

\subsection{A catalogue indexed by the number of leaves}

We write $\Phi_n$ for the continuous mappings $\varphi:[0,\infty)\rightarrow\mathbb{R}$ with $\varphi(0)=1$ such that $(\bm{t},\bm{s})\mapsto\varphi(\|\bm{t}-\bm{s}\|)$ is a correlation
function on $\mathbb{R}^n$, and $\Phi_\infty=\bigcap_{n\geq1}\Phi_n$. These classes decrease as $n$ grows, so that $\Phi_\infty$ is the smallest of them. The main result of this section is that composing any $\varphi\in\Phi_\kappa$ with $h_\kappa$ gives a valid covariance function on the tree, where $\kappa$ is the number of leaves.

\begin{theorem}
\label{thm:pd}
Let $\mathbb{T}$ be a compact metric tree with $\kappa$ leaves and let $h_\kappa$ be given by
\eqref{eq:h}. Then, for every $\varphi\in\Phi_\kappa$ and every $\sigma^2>0$,
\begin{equation}
\label{eq:model}
C(s,t) = \sigma^2\,\varphi\big(h_\kappa(s,t)\big), \qquad s,t\in\mathbb{T},
\end{equation}
is a covariance function on $\mathbb{T}$.
\end{theorem}

The class that appears is $\Phi_\kappa$ because the proof places the tree inside $\mathbb{R}^\kappa$. The class is strictly larger than $\Phi_\infty$ and contains the compactly supported correlations and those that attain negative values. It also becomes larger as $\kappa$ decreases.

In particular every member of $\Phi_\infty$ is admissible on any metric tree. Table
\ref{tab:phi} lists three of its best known members, the powered exponential, the Mat\'ern and
the generalized Cauchy, with the conditions that define each of them in Euclidean spaces
\citep{cressie2015statistics} and no further restriction. The table also includes the cardinal
sine, the spherical and the Askey models, whose admissibility depends on the number of leaves,
the first two requiring $\kappa\leq3$.

\begin{table}
\tbl{Correlation models on a metric tree with $\kappa$ leaves. Membership of $\Phi_\kappa$ is
what Theorem \ref{thm:pd} requires, and the first three families belong to $\Phi_\infty$ and are
admissible on every tree. Here $\Gamma$ and $K_\nu$ are the gamma function and the modified
Bessel function of the second kind of order $\nu$, and $(x)_+=\max(x,0)$.}{\setlength{\tabcolsep}{12pt}
\renewcommand{\arraystretch}{1.45}
\begin{tabular}{@{}lll@{}}
\toprule
Family & $\varphi(r)$ & Conditions for $\varphi\in\Phi_\kappa$ \\
\midrule
Powered exponential & $\exp\big(-(r/a)^{\alpha}\big)$ & $a>0$, $0<\alpha\leq 2$ \\
Mat\'ern & $\big(2^{1-\nu}/\Gamma(\nu)\big)(r/a)^{\nu}K_\nu(r/a)$ & $a>0$, $\nu>0$ \\
Generalized Cauchy & $\big(1+(r/a)^{\alpha}\big)^{-\beta/\alpha}$ & $a>0$, $0<\alpha\leq 2$, $\beta>0$ \\
\midrule
Cardinal sine & $\sin(r/a)/(r/a)$ & $a>0$, $\kappa\leq 3$ \\
Spherical & $\big(1-3(r/a)/2+(r/a)^{3}/2\big)_{+}$ & $a>0$, $\kappa\leq 3$ \\
Askey & $(1-r/a)_{+}^{\tau}$ & $a>0$, $\tau\geq(\kappa+1)/2$ \\
\bottomrule
\end{tabular}}
\label{tab:phi}
\end{table}

Under the shortest-path metric the situation is different. Most of the models of Table
\ref{tab:phi} are not covariance functions when written in terms of $d$. \cite{anderes2020}
identified which members of the first three families are, namely the Mat\'ern with
$0<\nu\leq 1/2$ and the powered exponential and the generalized Cauchy with $0<\alpha\leq 1$.
These are the roughest members of each family, and the corresponding fields are nowhere
differentiable.

\subsection{Behaviour of the field along the edges}
\label{sec:edges}

 {The two metrics agree inside every edge, by Proposition \ref{thm:struct}, so} the restriction of \eqref{eq:model} to a single edge is $\sigma^2\varphi(d(s,t))$, which under the arc-length parametrization is the classical stationary model on an interval. The Euclidean theory therefore describes the local behaviour of the field along each edge \citep{stein1999,gneiting2004}.

\begin{corollary}
\label{cor:local}
Consider the model \eqref{eq:model}, and let $s$ and $t$ lie in a common edge. Under the Mat\'ern model the field is $m$ times mean square differentiable along that edge if and only if $\nu>m$. Its sample paths there have fractal index $2\nu$ for $\nu<1$ under the Mat\'ern model, and $\alpha$ under the powered exponential and generalized Cauchy models.
\end{corollary}

In Euclidean spaces, the coefficient $\beta$ of the generalized Cauchy family describes the decay of the correlation at arbitrarily large separations (Hurst parameter) \citep{gneiting2004}. The trees considered in this paper are bounded, so that interpretation is not available.

Figure~\ref{fig:covreal} shows the covariance on a tree with three leaves, for three members of
the Mat\'ern family in terms of $h_\kappa$. Here $a$ denotes the practical range rather than the
scale of Table \ref{tab:phi}, that is, $r/a$ is replaced by $c_\nu\,r/a$, where $c_\nu$ solves
$(2^{1-\nu}/\Gamma(\nu))c^{\nu}K_\nu(c)=0.05$, so that the correlation equals $0.05$ at $r=a$
and holding $a$ fixed leaves the smoothness as the only difference between the three. At
$\nu=1/2$ the surface has a sharp peak at the point marked in red and decreases linearly with
the distance near it, while at $\nu=3/2$ and $\nu=5/2$ it is smooth there and decreases
quadratically at short distances, which is the local behaviour described in Corollary
\ref{cor:local}. The last two are not admissible under the shortest-path metric.

\begin{figure}
\centering
\includegraphics[width=\textwidth]{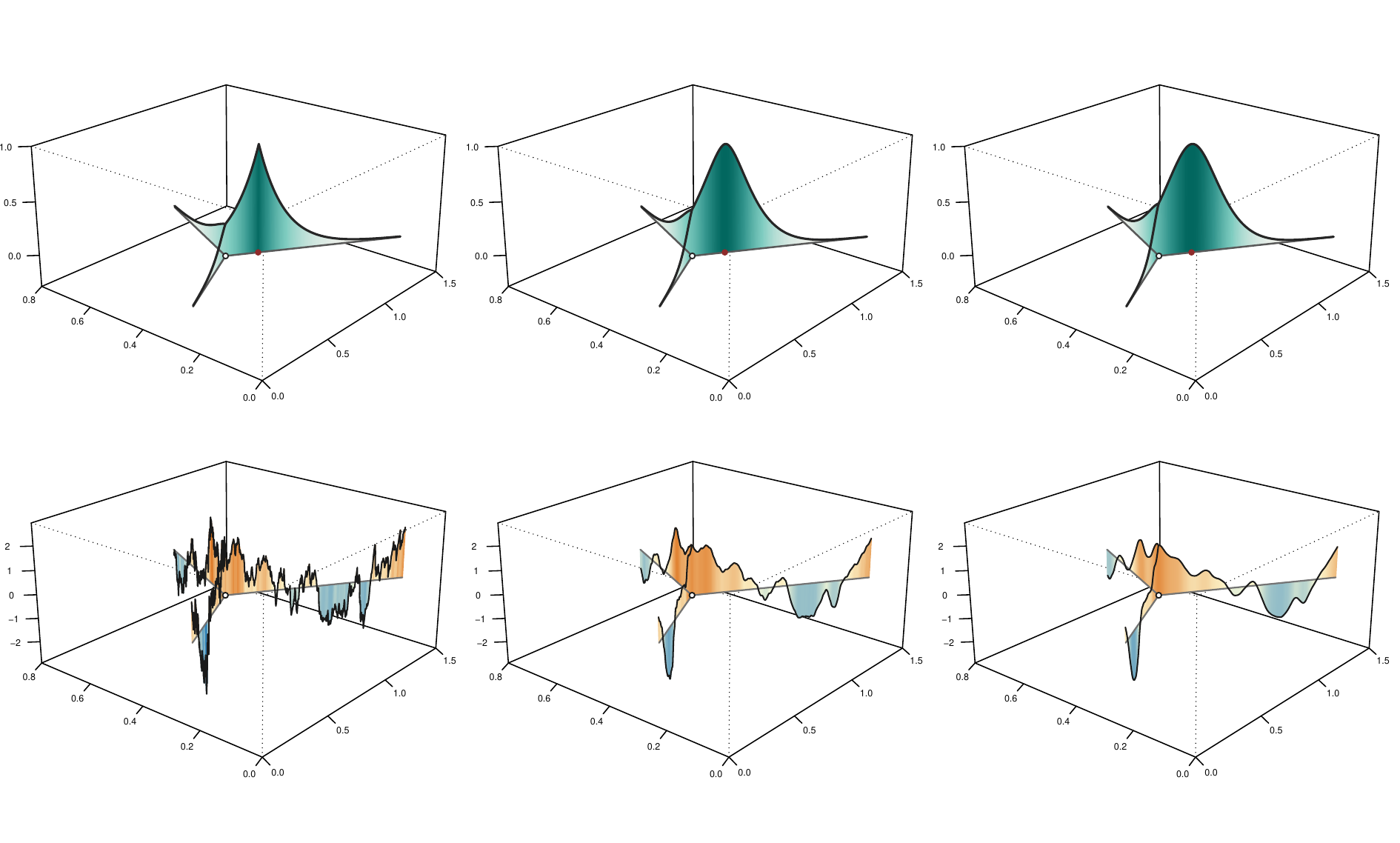}
\caption{The Mat\'ern model on a tree with three leaves, for $\nu=1/2$, $3/2$ and $5/2$, from
left to right. Top, the covariance $t\mapsto\varphi(h_\kappa(s,t))$ with $s$ at the red point.
Bottom, a realization of the field. The tree is drawn at zero height, and all six panels share
the same practical range, the three realizations sharing also the same random seed.}
\label{fig:covreal}
\end{figure}

The lower row of the same figure shows a realization of the field under each of the three
models, the three panels sharing the same random seed. At $\nu=1/2$ the sample path varies erratically along each edge;
at $\nu=3/2$ it is once differentiable and at $\nu=5/2$ twice, and the large-scale features are
preserved while the roughness decreases.

A tree with three leaves describes a single branching point, as at the confluence of two
streams or at the bifurcation of an axon or of a root. Figure~\ref{fig:covtree} shows the
cardinal sine and the spherical models on such a tree. The cardinal sine changes sign, so the
covariance is negative at intermediate distances and returns to zero at larger ones, which is
the hole effect. The spherical model is never negative and is zero beyond a finite distance.
These two shapes are not found in $\Phi_\infty$, and on a tree with a branching point neither
model is admissible under the shortest-path metric.

\begin{figure}
\centering
\includegraphics[width=0.92\textwidth]{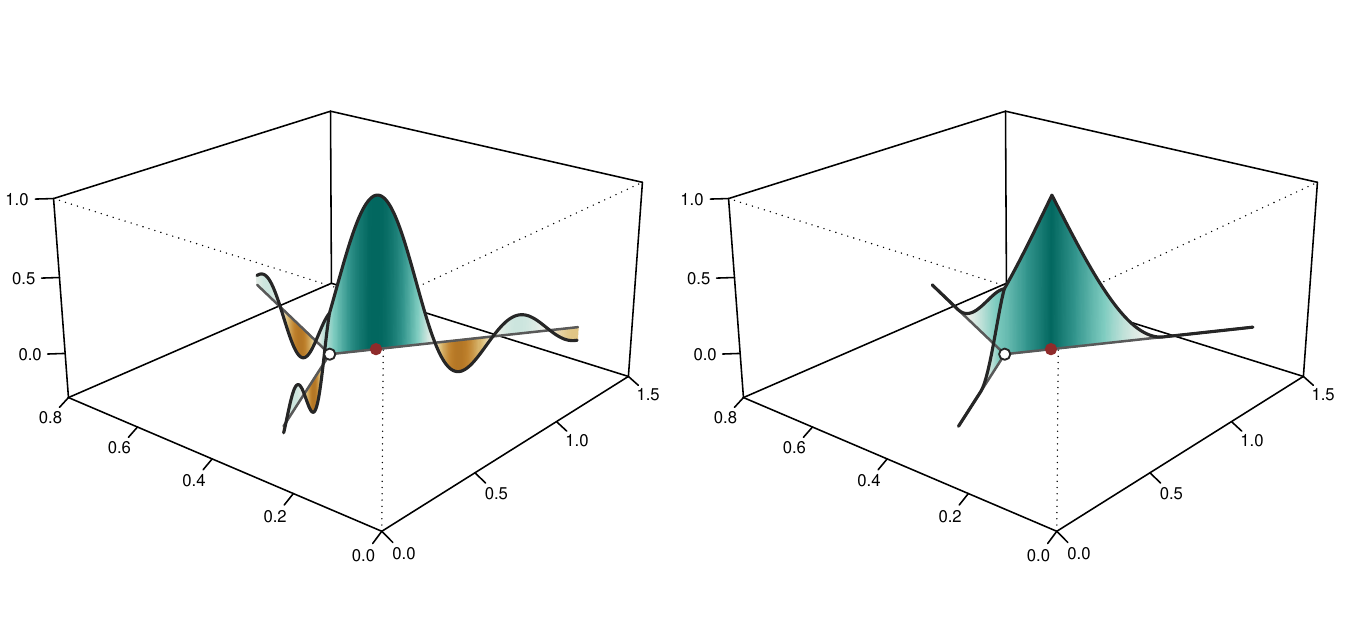}
\caption{The covariance $t\mapsto\varphi(h_\kappa(s,t))$ on a tree with three leaves, under two
models of Table \ref{tab:phi} that do not belong to $\Phi_\infty$. The tree is drawn at zero
height, the red point is $s$ and the white point is the vertex at which the tree branches. Left,
the cardinal sine, which takes negative values along the parts shown in brown. Right, the
spherical model, which vanishes beyond a finite distance.}
\label{fig:covtree}
\end{figure}

\subsection{Multivariate and space-time models}
\label{sec:ext}

{The results of this subsection complement Theorem \ref{thm:pd}.} The construction does not require the covariance to be scalar or purely spatial, since any statement of nonnegative definiteness available in $\mathbb{R}^\kappa$ holds on $\mathbb{T}$ as well. Two cases of practical interest follow. In the first, the
field is $p$-variate, $\bm{Z}=(Z_1,\hdots,Z_p)^\top$, and the object to be specified is the
matrix-valued covariance function whose $(i,j)$ entry is $\mathrm{cov}\{Z_i(s),Z_j(t)\}$. In the
second, the field is indexed by a point of the tree and a time, $\{Z(s,u),\,s\in\mathbb{T},\,
u\in\mathbb{R}\}$, and the object to be specified is a covariance function on
$\mathbb{T}\times\mathbb{R}$.

\begin{corollary}
\label{cor:multi}
Let $\bm{\varphi}=(\varphi_{ij})_{i,j=1}^{p}$ be a mapping from $[0,\infty)$ into the $p\times p$
real matrices such that $(\bm{x},\bm{y})\mapsto\bm{\varphi}(\|\bm{x}-\bm{y}\|)$ is a matrix-valued
covariance function on $\mathbb{R}^\kappa$. Then
\begin{equation}
\label{eq:multi}
\bm{C}(s,t) \;=\; \bm{\varphi}\big(h_\kappa(s,t)\big), \qquad s,t\in\mathbb{T},
\end{equation}
is a matrix-valued covariance function on $\mathbb{T}$, so that there exists a $p$-variate
random field $\bm{Z}=(Z_1,\hdots,Z_p)^\top$ on $\mathbb{T}$ with
$\mathrm{cov}\{Z_i(s),Z_j(t)\}=\varphi_{ij}(h_\kappa(s,t))$ for every $i,j=1,\hdots,p$.
\end{corollary}

The multivariate Mat\'ern model of \cite{gneiting2010matern} is valid in \eqref{eq:multi} under the conditions that make it a covariance function on $\mathbb{R}^\kappa$. These conditions
depend on the dimension, so the number of leaves determines which parameter values are admissible,
and there is no upper bound on the smoothness parameters $\nu_{ij}$. The bivariate case admits an
explicit statement.

\begin{example}
\label{ex:bimatern}
Let $p=2$ and write $\mathcal{M}(r\,|\,\nu,a)=\big(2^{1-\nu}/\Gamma(\nu)\big)(r/a)^{\nu}K_\nu(r/a)$
for the Mat\'ern correlation of Table \ref{tab:phi} with smoothness $\nu$ and scale $a$. In the
bivariate Mat\'ern model each entry of \eqref{eq:multi} is a Mat\'ern correlation with its own
smoothness and its own scale,
\[
\bm{C}(s,t)=
\begin{pmatrix}
\sigma_1^2\,\mathcal{M}(h_\kappa(s,t)\,|\,\nu_{11},a_{11}) &
\rho\,\sigma_1\sigma_2\,\mathcal{M}(h_\kappa(s,t)\,|\,\nu_{12},a_{12})\\[3pt]
\rho\,\sigma_1\sigma_2\,\mathcal{M}(h_\kappa(s,t)\,|\,\nu_{12},a_{12}) &
\sigma_2^2\,\mathcal{M}(h_\kappa(s,t)\,|\,\nu_{22},a_{22})
\end{pmatrix},
\]
where $\sigma_1^2,\sigma_2^2>0$ and $\rho\in[-1,1]$. Part (d) of Theorem 3 of
\cite{gneiting2010matern} covers the case
\begin{equation}
\label{eq:bimatern}
\nu_{12}=\tfrac{1}{2}(\nu_{11}+\nu_{22}),
\qquad
a_{12}\;\leq\;\min(a_{11},a_{22}),
\end{equation}
in which the model is a matrix-valued covariance function on $\mathbb{R}^\kappa$, and therefore on
$\mathbb{T}$ by Corollary \ref{cor:multi}, if and only if
\begin{equation}
\label{eq:bimatern2}
|\rho| \;\leq\;
\left(\frac{a_{12}}{a_{11}}\right)^{\!\nu_{11}}
\left(\frac{a_{12}}{a_{22}}\right)^{\!\nu_{22}}
\left\{\frac{\Gamma(\nu_{11}+\kappa/2)}{\Gamma(\nu_{11})}\,
\frac{\Gamma(\nu_{22}+\kappa/2)}{\Gamma(\nu_{22})}\right\}^{1/2}
\frac{\Gamma(\nu_{12})}{\Gamma(\nu_{12}+\kappa/2)} .
\end{equation}
The number of leaves enters \eqref{eq:bimatern2} only through the gamma factors, which equal one
when $\nu_{11}=\nu_{22}$ and decrease as $\kappa$ grows and as the two marginal smoothness
parameters separate. The scales contribute a factor that does not depend on $\kappa$ and is at most
one by \eqref{eq:bimatern}. Equal scales and equal smoothness
place no restriction on $\rho$, and setting $a_{11}=a_{22}=a_{12}$ gives the parsimonious model,
whose bound does not involve the scale.

\end{example}

\begin{corollary}
\label{cor:st}
Let $\widetilde{\varphi}$ be a mapping from $[0,\infty)^2$ into $\mathbb{R}$ such that
$\big((\bm{x},u),(\bm{y},w)\big)\mapsto\widetilde{\varphi}(\|\bm{x}-\bm{y}\|,|u-w|)$ is a covariance
function on $\mathbb{R}^\kappa\times\mathbb{R}$. Then
\begin{equation}
\label{eq:st}
C\big((s,u),(t,w)\big) \;=\; \widetilde{\varphi}\big(h_\kappa(s,t),|u-w|\big),
\qquad (s,u),(t,w)\in\mathbb{T}\times\mathbb{R},
\end{equation}
is a covariance function on $\mathbb{T}\times\mathbb{R}$, so that there exists a
random field $\{Z(s,u),\,s\in\mathbb{T},\,u\in\mathbb{R}\}$ with
$\mathrm{cov}\{Z(s,u),Z(t,w)\}=\widetilde{\varphi}\big(h_\kappa(s,t),|u-w|\big)$.
\end{corollary}

The class of \cite{gneiting2002nonseparable} gives a concrete family. If $\eta$ is completely monotone on $[0,\infty)$ and $\psi$ is positive with a completely monotone derivative, two conditions for which \cite{gneiting2002nonseparable} provides definitions and examples, then
$$C\big((s,u),(t,w)\big) \;=\; \frac{\sigma^2}{[\psi\big(h_\kappa^2(s,t)\big)]^{1/2}}\,
\eta\!\left(\frac{(u-w)^2}{\psi\big(h_\kappa^2(s,t)\big)}\right) $$
is a covariance function on $\mathbb{T}\times\mathbb{R}$ for every metric tree, with no restriction on the number of leaves. Writing the two arguments in the opposite roles, as in \cite{gneiting2002nonseparable}, would give the
exponent $\kappa/2$. In this form $\psi$ governs the purely
spatial behaviour and $\eta$ the purely temporal one.

\begin{example}
\label{ex:st}
A concrete choice is $\eta(r)=\exp(-c\,r^{\gamma})$, with $c>0$ and $0<\gamma\leq1$, together with $\psi(r)=\{1+(r/b^{2})^{\alpha}\}^{\beta}$, with $b>0$,
$0<\alpha\leq1$ and $0<\beta\leq1$. Both
appear in Tables 1 and 2 of \cite{gneiting2002nonseparable}. They give
$$ C\big((s,u),(t,w)\big) \;=\; \frac{\sigma^2}
{[1+\{h_\kappa(s,t)/b\}^{2\alpha}]^{\beta/2}}\,
\exp\!\left(-\,\frac{c\,|u-w|^{2\gamma}}
{[1+\{h_\kappa(s,t)/b\}^{2\alpha}]^{\beta\gamma}}\right),$$
which is a covariance function on $\mathbb{T}\times\mathbb{R}$ for every metric tree. Setting
$u=w$ gives the generalized Cauchy correlation of Table \ref{tab:phi} in the spatial argument, and setting $s=t$
gives the powered exponential correlation in time. The parameter $\beta$
governs the interaction between the two arguments, and the limiting value $\beta=0$ gives the
separable model, the product of those two correlations.
\end{example}

Corollaries \ref{cor:multi} and \ref{cor:st} are proved in the Appendix, together with Theorem
\ref{thm:pd}. Corollary \ref{cor:local} applies unchanged to \eqref{eq:multi} and \eqref{eq:st},
so the smoothness along each edge is available there too. Sections~\ref{sec:sim} and~\ref{sec:data} keep to the univariate and purely spatial case, where the comparison with the
shortest-path metric is most direct.

\section{Simulation study}
\label{sec:sim}

This section studies the models of Section~\ref{sec:cov} on simulated data. We first examine the
accuracy with which maximum likelihood estimates the parameters of the Mat\'ern model, and then
the consequences of fitting each metric to data generated from the other.

Throughout we work on the tree of Figure~\ref{fig:tree}, which has six edges, $\kappa=4$ leaves and
total length $2.78$, observed at $278$ points placed along the edges at a spacing of $0.01$. The
largest value of $h_\kappa$ over that configuration is $1.47$. 

\subsection{Maximum likelihood estimation}
\label{sec:sim-mle}

Here $a$ is the practical range of the Mat\'ern correlation, as in Section~\ref{sec:edges}. The two
parametrizations differ by the factor $c_\nu$, which depends only on $\nu$, so the relative
error reported below is the same under either one.
The field is Gaussian and {has covariance} \eqref{eq:model}, with zero mean, unit variance, no nugget effect and the Mat\'ern correlation above, and each realization is obtained from the Cholesky factor of the covariance matrix.
Nine scenarios combine $\nu\in\{1/2,3/2,5/2\}$ with $a\in\{0.2,0.6,1.0\}$. Each
scenario is replicated $100$ times, and $\sigma^2$, $a$ and $\nu$ are estimated jointly by maximum
likelihood.

Figure~\ref{fig:sim} reports the relative error of each estimate. They are all
{centred} at zero, so the three parameters are estimated without appreciable bias at every
degree of smoothness and at every range, including $\nu=3/2$ and $\nu=5/2$, for which no model is
available under the shortest-path metric.

The range affects the three parameters in opposite directions. The errors of $\sigma^2$ and of $a$
grow as $a$ grows, which is the familiar deterioration of likelihood estimation when the range
approaches the size of the domain. The errors of $\nu$ are the smallest throughout and shrink
instead, both as the field becomes smoother and as the range grows, because a long range places
most pairs of points at separations where the behaviour of the correlation near the origin governs
the likelihood. 

\begin{figure}
\centering
\includegraphics[width=\textwidth]{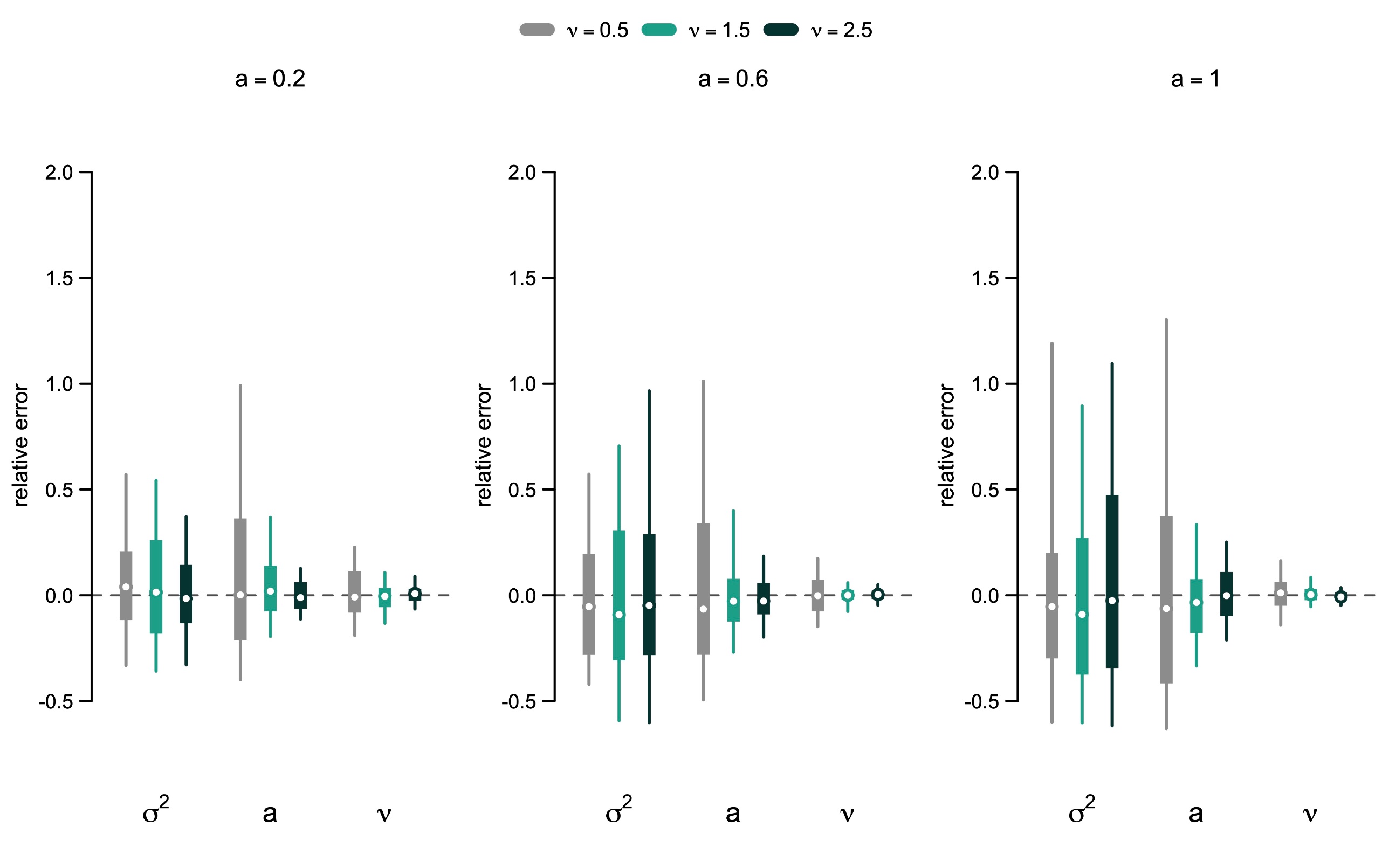}
\caption{Relative error $(\hat\theta-\theta)/\theta$ of the maximum likelihood estimates of the variance
$\sigma^2$, the range $a$ and the smoothness $\nu$, over $100$ replications. The point is the
median, the thick segment the central half of the distribution and the thin one its central ninety
per cent. The dashed line marks the absence of error.}
\label{fig:sim}
\end{figure}

\subsection{Misspecification study}
\label{sec:sim-mis}

The design is the one
of Section~\ref{sec:sim-mle}, with the field generated from $h_\kappa$ in nine scenarios and from
the shortest-path metric in three more, the latter restricted to $\nu=1/2$, the maximum admissible
value there. Each realization is then fitted twice, once with $h_\kappa$ and $\nu$ free, and once
with the shortest-path metric and $\nu=1/2$, so that the second model estimates two parameters and
the first three. Table \ref{tab:mis} reports the mean Akaike information criterion of each fit,
their mean difference and the percentage of replications in which the model built from $h_\kappa$
is selected.

When the field is smooth, the shortest-path metric is
unable to represent it and loses by hundreds of points of the criterion, by more as the field
becomes smoother and as the range grows, and it is never selected. When the field is generated
from the shortest-path metric, the model built from $h_\kappa$ loses by about two points, which is
the price of its additional parameter, so that the two fits are equivalent in likelihood and the
criterion simply prefers the smaller model. The same happens at $\nu=1/2$ under $h_\kappa$, where
the two metrics describe the data equally well. The cost of using $h_\kappa$ when the shortest-path metric would have sufficed is therefore
one additional parameter, and the cost of using the shortest-path metric on a smooth field is a
model that cannot describe it.

The last column shows that the misspecification does not affect the smoothness. When the field is
generated from the shortest-path metric and fitted with $h_\kappa$, the median estimate of $\nu$ is $1/2$
in all three scenarios. The two metrics agree inside every edge, and the smoothness is identified
from the behaviour of the correlation at short separations, so it is unaffected by a metric that
differs only across branching points.

\begin{table}
\tbl{Mean AIC of the two fitted models over $100$ replications, their mean difference, the
percentage of replications in which the model built from $h_\kappa$ is selected, and the median of the estimates of $\nu$ under $h_\kappa$. The field is
generated from $h_\kappa$ in the upper block and from the shortest-path metric in the lower one.
The model fitted under $h_\kappa$ estimates $\sigma^2$, $a$ and $\nu$; the one fitted under $d$
estimates $\sigma^2$ and $a$ with $\nu=1/2$, the maximum admissible value.}{\begin{tabular}{@{}llrrrrrr@{}}
\toprule
generated & $\nu$ & $a$ & AIC under $h_\kappa$ & AIC under $d$ & difference & selected (\%) &  $\hat\nu$ \\
\midrule
$h_\kappa$ & $1/2$ & 0.2 & $421.0$ & $420.2$ & $-0.781$ & 28 & $0.496$ \\
 &  & 0.6 & $133.4$ & $132.6$ & $-0.780$ & 25 & $0.496$ \\
 &  & 1.0 & $-7.17$ & $-7.621$ & $-0.453$ & 33 & $0.500$ \\
 & $3/2$ & 0.2 & $-304.7$ & $-73.27$ & $231.4$ & 100 & $1.501$ \\
 &  & 0.6 & $-1148$ & $-662.4$ & $485.8$ & 100 & $1.500$ \\
 &  & 1.0 & $-1561$ & $-957.5$ & $603.8$ & 100 & $1.504$ \\
 & $5/2$ & 0.2 & $-926.9$ & $-220.1$ & $706.7$ & 100 & $2.520$ \\
 &  & 0.6 & $-2301$ & $-829.2$ & $1472$ & 100 & $2.505$ \\
 &  & 1.0 & $-2973$ & $-1129$ & $1844$ & 100 & $2.482$ \\
\midrule
$d$ & $1/2$ & 0.2 & $418.2$ & $416.8$ & $-1.379$ & 17 & $0.503$ \\
 &  & 0.6 & $133.0$ & $131.2$ & $-1.821$ & 13 & $0.503$ \\
 &  & 1.0 & $-1.647$ & $-3.179$ & $-1.531$ & 22 & $0.500$ \\
\bottomrule
\end{tabular}}
\label{tab:mis}
\end{table}

\section{Application to the local radius of a microglial cell}
\label{sec:data}

The cell of Figure~\ref{fig:data} is a microglial cell from the main olfactory bulb of an adult female mouse, reconstruction NMO\_210282 of the Siegert archive at NeuroMorpho.Org \citep{ascoli2007}. It was produced by immunostaining for Iba1 and confocal imaging of a
$100\,\mu$m coronal slice, and traced with Imaris. The shape of microglial cells is itself the subject of the study from which the reconstruction comes \citep{colombo2022}, and the animal is from a model of Alzheimer-like degeneration.

\begin{figure}
\centering
\includegraphics[width=0.5\textwidth]{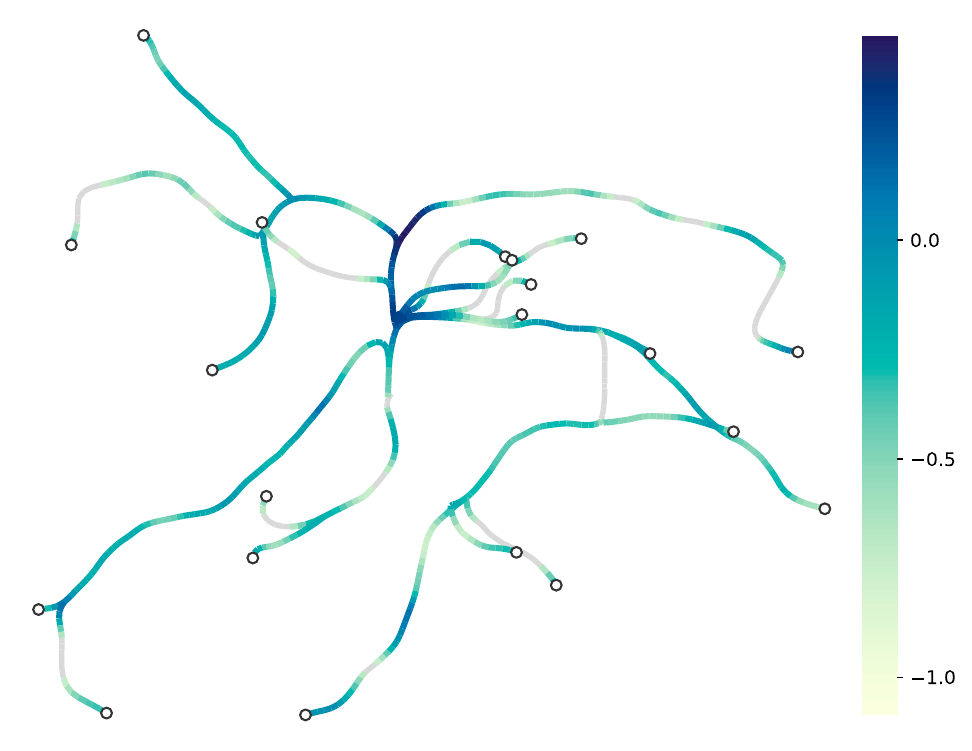}
\caption{The reconstructed microglial cell. Colour gives the logarithm of the local radius in micrometres. The grey parts are the nodes carrying the smallest recorded radius, discarded as a detection limit of the tracing, and the circles mark the $\kappa=20$ leaves.}
\label{fig:data}
\end{figure}

The reconstruction is supplied in the SWC format, in which each node records its position, its local radius and the identity of its parent node, so that the parent pointers define the topology. We reserve the word node for these recording points, and the word vertex for
the branching points and the leaves of the metric tree; a node that is neither a branching point nor
a leaf is an interior point of an edge. The tree has $723$ nodes and $20$ leaves, and spans $245\,\mu$m of process with an average diameter of $1.47\,\mu$m. We use the projection of
the reconstruction onto the $xy$ plane, which preserves the topology and reduces the total length by about fifteen per cent, so that the ranges reported below are in micrometres in that plane.

Ninety-six nodes carry the smallest recorded radius, $0.450\,\mu$m; the next distinct value is $0.452\,\mu$m. They form twenty runs distributed over the tree rather than concentrated at the tips, which indicates a detection limit of the tracing rather than a genuine calibre. We discard them, leaving $627$ nodes and twenty portions of the tree without data, which are predicted below. Write $r(s)$ for the local
radius at $s\in\mathbb{T}$, so that the data are the values $r(s_1),\hdots,r(s_n)$ at the $n=627$
retained nodes $s_1,\hdots,s_n$. The radius is positive and varies smoothly along each process, and at a branching point the daughter processes scale relative to the parent, a relation described by \citet{rall1959}. Growth along the tree is therefore multiplicative rather than additive, which places the natural
model on the logarithmic scale. The
marginal distribution of the data supports the choice, since the radius has skewness $1.00$ and
kurtosis $4.33$ while its logarithm has $0.20$ and $2.82$. We model $Y(s)=\log r(s)$, $s\in\mathbb{T}$, as a Gaussian random field.

We fit $Y(s) = \mu + Z(s)$, with the covariance structure
\begin{equation}
\label{eq:fit}
\mathrm{cov}\{Z(s),Z(t)\}  =  \sigma^2\big\{(1-\gamma)\,\varphi\big(\delta(s,t)\big) + \gamma\,\mathbb{I}[s=t]\big\}, \quad s,t\in\mathbb{T},
\end{equation}
with $\delta$ either $h_{\kappa}$ or $d$, and $\varphi$ the Mat\'ern correlation of Table \ref{tab:phi} with $\nu$ fixed. The parameters $\mu$, $\sigma^2$, $a$ and the nugget fraction $\gamma$ are estimated by maximum likelihood. Three choices of $\nu$ are used with $h_{\kappa}$, namely $1/2$, $3/2$ and $5/2$. With $d$ only $\nu=1/2$ is used, and that model is the benchmark.

Predictive performance is measured by leave-one-out cross-validation, each node being predicted from the remaining ones by ordinary kriging. As a reference we also report the error obtained by predicting each node by the average of the others.

Table \ref{tab:fit} collects the fits and the cross-validated errors. First, it shows that the variable is strongly structured along the tree, since any of the four models reduces the prediction error by an order of magnitude relative to the uncorrelated reference. The benchmark based on the shortest-path distance already achieves most of this reduction, which is expected, since the exponential model is a reasonable description of a quantity that varies continuously along
each process.

\begin{table}
\tbl{Mat\'ern models fitted by maximum likelihood with standard errors in parentheses, and their leave-one-out prediction errors. The nugget effect
is estimated at zero in every model and is not shown. The practical range is in micrometres and the remaining entries refer to the logarithm of the radius.}{\begin{tabular}{@{}llcccccc@{}}
\toprule
distance & $\nu$ & mean & variance  & practical range & log-likelihood & MAE & RMSE \\
\midrule
$h_\kappa$ & $1/2$  & $-0.212$ (0.173)  & 0.135 (0.057)  & 24.5 (10.6) & 1165.6 & 0.0283 & 0.0433 \\
$h_\kappa$ & $3/2$  & $-0.361$ (0.079)  & 0.181 (0.040)  &  6.27 (0.54) & 1433.5 & 0.0122 & 0.0205 \\
$h_\kappa$& $5/2$ & $-0.383$ (0.048)   & 0.129 (0.020)  &  3.63 (0.17) & 1478.0 & 0.0114 & 0.0199 \\
$d$ & $1/2$   & $-0.421$ (0.051)   & 0.066 (0.011)  & 11.4 (1.94) & 1145.1 & 0.0294 & 0.0464 \\
\midrule
\multicolumn{6}{@{}l}{no model, prediction by the sample mean} & 0.2143 & 0.2701 \\
\bottomrule
\end{tabular}}
\label{tab:fit}
\end{table}

Second, a smooth model reduces the remaining error by about half. Passing from $\nu=1/2$ to $\nu=3/2$ under $h_{\kappa}$ increases the log-likelihood by $268$ and reduces the root mean squared error from $0.043$ to $0.021$; a further increase to $\nu=5/2$ produces a smaller gain. The
improvement over the benchmark is of the same magnitude, from $0.046$ to $0.020$. The range is also better determined under the smooth models, since its standard error falls from $10.6$ at $\nu=1/2$ to $0.17$ at $\nu=5/2$.

Figure~\ref{fig:krig} maps the predicted field over the whole tree under the Mat\'ern model with
$\nu=5/2$, including the portions described above. The standard error of prediction is small at
the nodes where the radius was recorded and increases along the discarded portions.

\begin{figure}
\centering
\includegraphics[width=\textwidth]{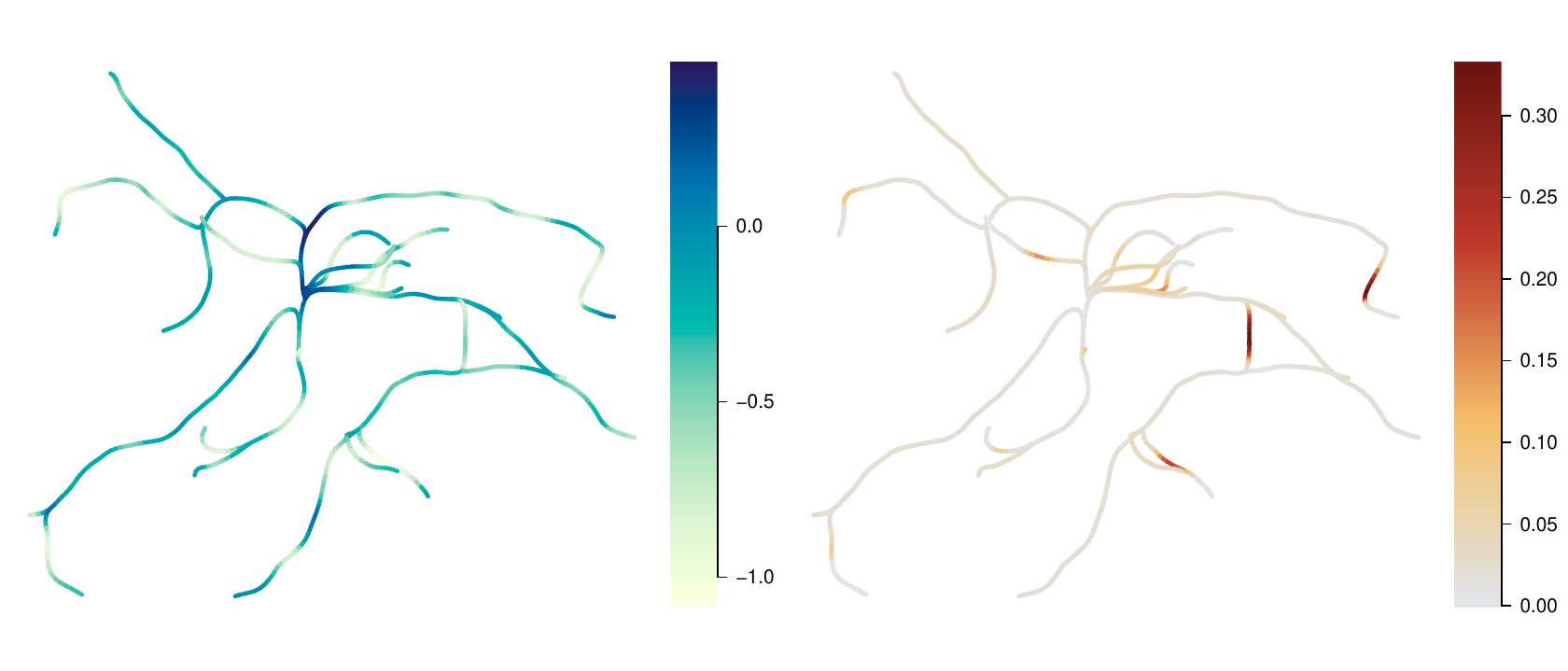}
\caption{Ordinary kriging of the logarithm of the local radius under the Mat\'ern model with
$\nu=5/2$. Left, the predicted field over the whole tree. Right, its standard error.}
\label{fig:krig}
\end{figure}

A resting microglial cell has many long thin processes, and when it is activated the processes
shorten and thicken, so that the degree of branching is a marker of the state of these cells in
injury, ageing and neurodegeneration \citep{colombo2022}. Quantifying that shape is the
measurement of interest, and a covariance model supplies it along the tree rather than in the
plane of the image, over portions where the radius was never recorded.

\section{Discussion}
\label{sec:discussion}

Validity on the tree requires a distance that places the tree inside a Euclidean space, whereas
the local behaviour of the field along an edge depends on how closely the distance follows the
length of the path. The metric built from the distances to the leaves places the tree inside
$\mathbb{R}^\kappa$, and inside every edge it follows that length exactly. The whole Euclidean
catalogue therefore becomes available on a tree, and the same argument gives matrix-valued and
space-time covariances. In the simulation study the parameters were estimated without
appreciable bias and the shortest-path metric was never selected once the field was smooth. On
the microglial cell of Section~\ref{sec:data} the smooth models reduce the cross-validated
prediction error by about half relative to the best model that the shortest-path metric allows.

Extending the construction to general metric graphs is the natural next step. Theorem \ref{thm:pd} remains valid on any metric graph, since its proof does not use the absence of cycles. The property that is lost is the agreement with the shortest-path metric inside each
edge. A reference point that can be reached from both ends of an edge makes the new distance strictly smaller than the shortest-path distance inside that edge, so that the restriction of the field to an edge is no longer the classical model of the Euclidean theory and Corollary \ref{cor:local} no longer applies. A first step is to measure that gap, in terms of the cycles that lie between a pair of points and the reference set, and to identify the graphs on which the local behaviour is still approximately the Euclidean one. Choosing the reference set is a second
question, since a graph with cycles may have no vertex of degree one, so that the leaves have no
immediate counterpart.

The restriction of the catalogue is not specific to graphs. The same obstacle arises on other non-Euclidean
domains, such as complex terrains in which the relevant distance follows the shape of the
domain and the catalogue of Table \ref{tab:phi} is again unavailable \citep{curriero2006}. On the sphere with the great-circle distance the same
pattern appears, as the Mat\'ern class survives only for small values of the smoothness and the powered exponential only for small values of its exponent \citep{10.3150/12-BEJSP06}. Comparing two points through their distances to a fixed set of reference points would make a large catalogue of covariance models available there as well. The device is the same in each case, that of keeping the covariance models and changing the way distances are measured, and what must be checked each time is the local behaviour of the new distance.

\section*{Acknowledgements}
Alfredo Alegr\'ia is partially supported by the National Agency for Research and Development of
Chile through grant ANID Fondecyt 1251154.

\section*{Code availability}
Code in \texttt{R} that evaluates the metric and illustrates simulation and maximum likelihood
estimation on a metric tree is available at \url{https://github.com/alfredoalegria/covariance-metric-trees}.

\section*{Appendix}

\subsection*{Notation and a preliminary lemma}

By a direction at a point $s\in\mathbb{T}$ we mean one of the ways of moving away from $s$ within $\mathbb{T}$. A point inside an edge has two of them, a vertex of degree $\ell$
has $\ell$, and a leaf has one. Moving away from $s$ along one direction and, at each vertex reached, continuing along an edge different from the one just used, one never returns to a point already visited, because $\mathbb{T}$ has no cycles, and, since $\mathbb{T}$ is a finite union
of edges, the walk stops at a vertex with no new edge to continue along, that is, at a leaf
$\xi$. Read backwards it is the path from $\xi$ to $s$, which is the only one and therefore has length $d(\xi,s)$. Every direction at a point thus leads to a leaf, a fact used repeatedly below.

Fix $s$ and $t$ and let $P$ be the path joining them. Since $\mathbb{T}$ has no cycles, the
path from a leaf $\xi_{j}$ to $s$ and the path from $\xi_{j}$ to $t$ reach $P$ at the same point
$m_j(s,t)$, and every path from $\xi_{j}$ to a point of $P$ passes through $m_j(s,t)$. Write $p_j(s,t)$ and
$p_j(t,s)$ for the lengths of the two pieces into which $m_j(s,t)$ splits $P$, so that they add to
$d(s,t)$, and, when $s\neq t$, write $x_j(s,t)=p_j(s,t)/d(s,t)$ for the position of $m_j(s,t)$ along
$P$ as a fraction of its length. Figure~\ref{fig:pj} shows the two cases that occur. The leaf
$\xi_{j}$ separates $s$ and $t$, in the sense of Section~\ref{sec:trees}, exactly when $m_j(s,t)$ is
interior to $P$, that is, when $0<x_j(s,t)<1$; and $m_j(s,t)$ is then a vertex of degree three or more.

\begin{figure}
\centering
\includegraphics[width=0.92\textwidth]{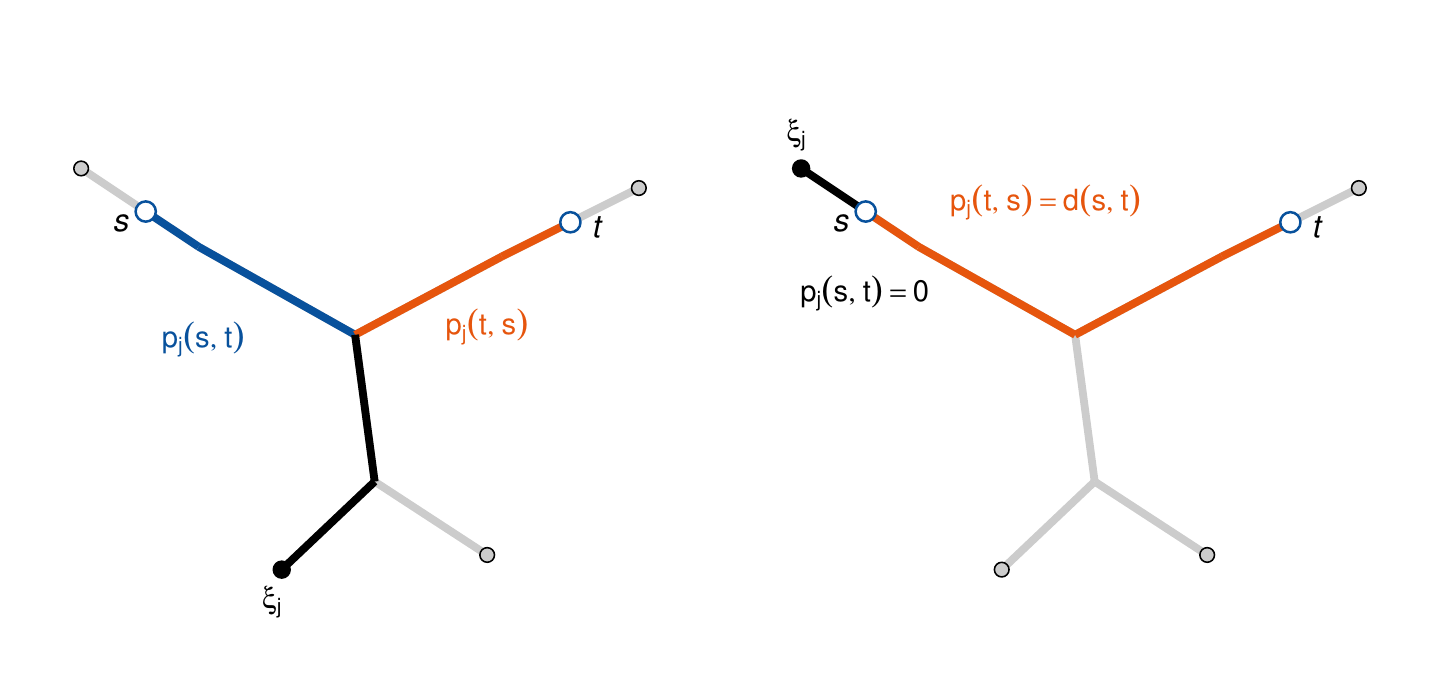}
\caption{A metric tree with four leaves and two of its points $s$ and $t$. In each panel the
path joining them is the coloured one and the path coming from the leaf $\xi_{j}$ is in black.
Left, that path reaching $P$ at an interior point, which splits it into the two pieces
$p_j(s,t)$ and $p_j(t,s)$, so that $\xi_{j}$ separates $s$ and $t$. Right, that path reaching
$P$ at $s$, so that $p_j(s,t)=0$ and $\xi_{j}$ does not separate them.}
\label{fig:pj}
\end{figure}

\begin{lemma}
\label{lem:id}
Let $s\neq t$ in $\mathbb{T}$. Then
\begin{equation}
\label{eq:xsq}
h_\kappa^2(s,t) = d^2(s,t)\left\{ 1 - \frac{4}{\kappa}\sum_{j=1}^{\kappa}
x_j(s,t)\{1-x_j(s,t)\} \right\} .
\end{equation}
\end{lemma}

\noindent\emph{Proof of Lemma \ref{lem:id}}. Every path from $\xi_{j}$ to a point of $P$
passes through $m_j(s,t)$, so that $d(\xi_{j},s)=d(\xi_{j},m_j(s,t))+p_j(s,t)$ and
$d(\xi_{j},t)=d(\xi_{j},m_j(s,t))+p_j(t,s)$. Subtracting, and using $p_j(s,t)+p_j(t,s)=d(s,t)$,
\[
d(\xi_{j},s) - d(\xi_{j},t) = p_j(s,t) - p_j(t,s) = \{2x_j(s,t)-1\}\,d(s,t) .
\]
Squaring, averaging over $j=1,\hdots,\kappa$ and using \eqref{eq:h},
$h_\kappa^2(s,t)=d^2(s,t)\kappa^{-1}\sum_{j=1}^\kappa\{2x_j(s,t)-1\}^2$, which is \eqref{eq:xsq} because
$(2x-1)^2=1-4x(1-x)$.
\qed

\subsection*{Proof of Proposition \ref{thm:struct}}

\begin{itemize}
\item[(i)] Write $\Delta(s,t)$ for the vector of $\mathbb{R}^\kappa$ with entries
$\kappa^{-1/2}\{d(s,\xi_j)-d(t,\xi_j)\}$, so that $h_\kappa(s,t)=\|\Delta(s,t)\|$ by \eqref{eq:h}.
Symmetry and nonnegativity are immediate, and the triangle inequality follows from the one for the
Euclidean norm, because $\Delta(s,t)=\Delta(s,u)+\Delta(u,t)$ for every $u\in\mathbb{T}$.
It remains to show that $h_\kappa(s,t)=0$ forces $s=t$, that is, that the distances to the leaves determine the point. Suppose $d(\xi_{j},s)=d(\xi_{j},t)$ for every leaf, and let $\xi_{a}$ and $\xi_{b}$ be the leaves reached along two directions at $s$, taking $\xi_{a}=s$ when $s$ is itself a leaf. Putting the two walks together gives the path from $\xi_{a}$ to $\xi_{b}$, so that
$d(\xi_{a},s)+d(\xi_{b},s)=d(\xi_{a},\xi_{b})$, and the hypothesis gives
$d(\xi_{a},t)+d(\xi_{b},t)=d(\xi_{a},\xi_{b})$, which holds only for points of the path joining
$\xi_{a}$ and $\xi_{b}$. Hence $t$ lies on that path, where $s$ already lies, and since the distance
to $\xi_{a}$ grows strictly along it and $d(\xi_{a},s)=d(\xi_{a},t)$, we conclude $s=t$.

\item[(ii)] Each term $x_j(s,t)\{1-x_j(s,t)\}$ of \eqref{eq:xsq} lies in
$[0,1/4]$, so the factor in braces lies in $[0,1]$, which gives the upper bound. For the lower bound, take at
$s$ a direction other than the one towards $t$, and let $\xi_{a}$ be the leaf reached along it.
The path from $\xi_{a}$ to $s$ reaches $P$ only at $s$, so that $m_a(s,t)=s$ and $x_a(s,t)=0$; if $s$ is a
leaf it has no such direction, and $\xi_{a}=s$ serves instead. Arguing in the same way at $t$
gives an index $b$ with $x_b(s,t)=1$. The terms of $\xi_{a}$ and $\xi_{b}$ vanish, so the sum in
\eqref{eq:xsq} is at most $(\kappa-2)/4$, the factor in braces is at least $2/\kappa$, and
$h_\kappa^2(s,t)\geq 2d^2(s,t)/\kappa$.

\item[(iii)] Equality holds in $h_\kappa(s,t)\leq d(s,t)$ precisely when every term of the sum in \eqref{eq:xsq} vanishes, that is, when $x_j(s,t)\in\{0,1\}$ for every $j=1,\hdots,\kappa$. Accordingly, if $m_j(s,t)$ is interior to $P$, then the path from $\xi_{j}$ meets $P$ at $m_j(s,t)$ and is not contained in $P$, so that $m_j(s,t)$ is a vertex of degree three or more. Hence, if no such vertex lies strictly between $s$ and $t$, then
$m_j(s,t)\in\{s,t\}$ and $x_j(s,t)\in\{0,1\}$ for every $j=1,\hdots,\kappa$, so that $h_\kappa(s,t)=d(s,t)$.
Conversely, let $b$ be a vertex of degree three or more lying strictly between $s$ and $t$. Two of the edges that meet at $b$ lie on $P$, and extending the path from $b$ along a third one until it can be extended no further produces a leaf whose path to $P$ meets it at $b$. For that leaf $x_j(s,t)$ lies strictly between $0$ and $1$, and $h_\kappa(s,t)<d(s,t)$. The interior of an edge contains no vertex, so the equality holds in particular when $s$ and $t$ lie
in a common edge.
 \end{itemize}
\qed

\subsection*{Proof of Proposition \ref{thm:local}}

A term of the sum in \eqref{eq:xsq} is zero unless $0<x_j(s,t)<1$, which
happens exactly for the $\kappa_{st}$ leaves that separate $s$ and $t$. Each summand is at most
$1/4$, so the sum is at most $\kappa_{st}/4$ and \eqref{eq:local} follows. Taking
$\kappa_{st}=\kappa-2$ returns the left-hand inequality of \eqref{eq:bilip}, and $\kappa_{st}=0$
returns $h_\kappa(s,t)\geq d(s,t)$, which with part (ii) gives the equality of part (iii).

\qed

\subsection*{Proof of Theorem \ref{thm:pd} and of Corollaries \ref{cor:multi} and \ref{cor:st}}

The proof uses a map under which $h_\kappa$ becomes a Euclidean distance. Define
$\Psi: \mathbb{T} \longrightarrow \mathbb{R}^{\kappa}$ by
\begin{equation}
\label{eq:map}
\Psi(s) = \frac{1}{\sqrt{\kappa}} \big( d(\xi_{1},s), \hdots, d(\xi_{\kappa},s) \big), \quad s\in\mathbb{T}.
\end{equation}
Comparing with \eqref{eq:h},
\begin{equation}
\label{eq:iso}
\|\Psi(s)-\Psi(t)\| = h_\kappa(s,t),
\qquad s,t\in\mathbb{T} .
\end{equation}
Given $s_1,\hdots,s_r\in\mathbb{T}$ and $a_1,\hdots,a_r\in\mathbb{R}$, write $\bm{t}_i=\Psi(s_i)$. By
\eqref{eq:iso},
$$\sum_{i=1}^r\sum_{q=1}^r a_i a_q\, \varphi\big( h_\kappa(s_i,s_q) \big) = \sum_{i=1}^r\sum_{q=1}^r
a_i a_q\, \varphi\big( \|\bm{t}_i - \bm{t}_q\| \big) \;\geq\; 0 ,$$
because $\varphi\in\Phi_\kappa$ and the $\bm{t}_i$ are points of $\mathbb{R}^\kappa$.

Corollaries \ref{cor:multi} and \ref{cor:st} follow by the same substitution. The quadratic form of \eqref{eq:multi} at $s_1,\hdots,s_r\in\mathbb{T}$ with vector-valued coefficients
$\bm{a}_1,\hdots,\bm{a}_r\in\mathbb{R}^p$, and the quadratic form of \eqref{eq:st} at
$(s_1,u_1),\hdots,(s_r,u_r)\in\mathbb{T}\times\mathbb{R}$ with real coefficients, become
$$\sum_{i=1}^r\sum_{q=1}^r \bm{a}_i^\top\, \bm{\varphi}\big( \|\bm{t}_i - \bm{t}_q\| \big)\, \bm{a}_q
\qquad\text{and}\qquad
\sum_{i=1}^r\sum_{q=1}^r a_i a_q\, \widetilde{\varphi}\big( \|\bm{t}_i - \bm{t}_q\|, |u_i-u_q| \big) .$$
The first is nonnegative because $\bm{\varphi}$ is a matrix-valued covariance function on
$\mathbb{R}^\kappa$ and the $\bm{t}_i$ are points of that space, and the second because $\widetilde{\varphi}$ defines a
covariance function on $\mathbb{R}^\kappa\times\mathbb{R}$ and the $(\bm{t}_i,u_i)$ are points of
that space.

\qed

\bibliographystyle{apalike}
\bibliography{matern}

\end{document}